\RequirePackage{ifpdf}
\ifpdf 
\documentclass[pdftex]{sigma}
\else
\documentclass{sigma}
\fi

\numberwithin{equation}{section}

\newcommand\nonu{\nonumber}

\newcommand\ZZ{\mathbb{Z}}

\newcommand\FSA{{\cal A}}

\newcommand\be\beta
\newcommand\ga\gamma
\newcommand\la\lambda

\newcommand\half{\frac12}
\newcommand\thalf{\tfrac12}
\newcommand{\qhyp}[5]{\,\mbox{}_{#1}\phi_{#2}\left(
  \genfrac{}{}{0pt}{}{#3}{#4};#5\right)}
\newcommand\LHS{left-hand side}
\newcommand\wt{\widetilde}

\newcommand\Asym{\FSA_{\rm sym}}
\newcommand\Dsym{D_{\rm sym}}
\newcommand\goH{\mathfrak{H}}
\newcommand\DAHA{\tilde\goH}

\newcommand\AWQ{\widetilde{AW}(3,Q_0)}
\newcommand\Psym{P_{\rm sym}}

\newtheorem{Theorem}{Theorem}[section]
\newtheorem{Lemma}[Theorem]{Lemma}
\newtheorem{Proposition}[Theorem]{Proposition}

\newtheorem{Remark}[Theorem]{Remark}
\newtheorem{Definition}[Theorem]{Definition}

\begin{document}

\renewcommand{\PaperNumber}{052}

\FirstPageHeading

\renewcommand{\thefootnote}{$\star$}

\ShortArticleName{Zhedanov's Algebra $AW(3)$. II}

\ArticleName{Zhedanov's Algebra $\boldsymbol{AW(3)}$  and the Double\\ Af\/f\/ine Hecke
Algebra  in the Rank One Case.\\
II. The Spherical Subalgebra\footnote{This
paper is a contribution to the Proceedings of the Seventh
International Conference ``Symmetry in Nonlinear Mathematical
Physics'' (June 24--30, 2007, Kyiv, Ukraine). The full collection
is available at
\href{http://www.emis.de/journals/SIGMA/symmetry2007.html}{http://www.emis.de/journals/SIGMA/symmetry2007.html}}}

\Author{Tom H. KOORNWINDER}

\AuthorNameForHeading{T.H. Koornwinder}

\Address{Korteweg-de Vries Institute, University of Amsterdam,\\
Plantage Muidergracht 24, 1018 TV Amsterdam, The Netherlands}
\Email{\href{mailto:thk@science.uva.nl}{thk@science.uva.nl}}
\URLaddress{\url{http://www.science.uva.nl/~thk/}}

\ArticleDates{Received November 15, 2007, in f\/inal form June 03,
2008; Published online June 10, 2008}

\Abstract{This paper builds on the previous paper by the author,
where a relationship
between Zhedanov's algebra $AW(3)$ and the double af\/f\/ine Hecke algebra
(DAHA)
corresponding to the Askey--Wilson polynomials was established.
It is shown here that the spherical subalgebra of this DAHA is
isomorphic to $AW(3)$ with an additional relation that the Casimir operator
equals an explicit constant.
A similar result with $q$-shifted parameters holds for the
antispherical subalgebra.
Some theorems on centralizers and centers for the
algebras under consideration will f\/inally be proved as corollaries of
the characterization of the spherical and antispherical subalgebra.}


\Classification{33D80}

\section{Introduction}
Zhedanov \cite{1} introduced in 1991 an algebra $AW(3)$
with three generators $K_0$, $K_1$, $K_2$ and
three relations in the form of $q$-commutators, which describes deeper
symmetries of the Askey--Wilson polynomials.
In the {\em basic representation} (or
{\em polynomial representation}) of $AW(3)$ on the space of symmetric
Laurent polynomials in $z$, $K_0$ acts as the
second order $q$-dif\/ference operator $\Dsym$ for which the
Askey--Wilson polynomials are eigenfunctions and $K_1$ acts as
multiplication by $z+z^{-1}$.
The Casimir operator $Q$ for $AW(3)$ becomes a scalar $Q_0$ in this
representation. Let $AW(3,Q_0)$ be $AW(3)$ with the additional relation
$Q=Q_0$. Then the basic representation $AW(3,Q_0)$ is faithful, see
\cite{17}.
There is a parameter changing anti-algebra isomorphism of $AW(3)$ which
interchanges $K_0$ and $K_1$, and hence interchanges $\Dsym$ and $z+z^{-1}$
in the basic representation. In the basic representation this {\em duality
isomorphism} can be realized by an integral transform having the
Askey--Wilson polynomial $P_n[z]$ as kernel
which maps symmetric Laurent polynomials
to inf\/inite sequences $\{c_n\}_{n=0,1,\ldots}$.

In 1992
Cherednik \cite{22} introduced double af\/f\/ine Hecke algebras associated
with root systems (DAHA's). This was the f\/irst of an important series
of papers by the same author, where a representation of the DAHA was
given in terms of $q$-dif\/ference-ref\/lection operators
($q$-analogues of Dunkl operators), joint eigenfunctions
of such operators were identif\/ied as non-symmetric Macdonald polynomials,
and Macdonald's conjectures for ordinary (symmetric) Macdonald polynomials
associated with root systems could be proved. The idea of nonsymmetric
polynomials was very fruitful, on the one hand as an important extension
of traditional harmonic analysis involving orthogonal systems of special
functions, on the other hand because the nonsymmetric point of view was
helpful for understanding the symmetric case better. For instance,
a duality anti-algebra isomorphism and shift operators (as studied
earlier by Opdam \cite{24} in the symmetric framework) occur naturally
in the non-symmetric context.

Related to Askey--Wilson polynomials the  DAHA of type $(C_1^\vee,C_1)$
(four parameters) was studied by Sahi \cite{10,25},
Noumi \& Stokman \cite{11}, and Macdonald \cite[Ch.~6]{5}.
See also the author's previous paper \cite{17}.
The same phenomena as described in the previous paragraph occur here, but
in a very explicit form. See for instance Remarks \ref{73} and \ref{74}
(referring to \cite{11}) about the duality anti-algebra isomorphism and
the shift operators, respectively.

In \cite{17} I also discussed how the algebra $AW(3,Q_0)$ is related to the
double af\/f\/ine Hecke algebra (DAHA) of type $(C_1^\vee,C_1)$.
In the basic (or polynomial) representation of this DAHA (denoted by $\DAHA$)
on the space of Laurent polynomials in $z$, the nonsymmetric
Askey--Wilson polynomials occur as eigenfunctions of a suitable element
$Y$ of $\DAHA$ (see \cite{10,11}, \cite[Ch.~6]{5}).
It turns out (see~\cite{17}) that a central extension $\AWQ$ of
$AW(3,Q_0)$ can be embedded as a subalgebra of $\DAHA$. As pointed out in
the present paper (see Remark \ref{73}), the duality anti-algebra isomorphisms
for $\AWQ$ and $\DAHA$ are compatible with this embedding.
\par
It would be interesting, also for possible generalisations to higher rank,
to have a more conceptual way of decribing the relationship between
Zhedanov's algebra and the double af\/f\/ine Hecke algebra.
The present paper establishes this by showing that the algebra
$AW(3,Q_0)$ is isomorphic to the spherical subalgebra of $\DAHA$.
The def\/inition of a {\em spherical subalgebra} of a
DAHA goes back to
Etingof \& Ginzburg \cite{18},
where a similar object was def\/ined in the context of Cherednik algebras,
see also \cite{19}. The def\/inition of spherical subalgebra for the
DAHA of type $(C_1^\vee,C_1)$
was given by Oblomkov \cite{20}. In general, the spherical subalgebra
of a DAHA $\DAHA$ is the algebra $\Psym\DAHA\Psym$, where
$\Psym$ is the symmetrizer idempotent in $\DAHA$.

The proof of the isomorphism of $AW(3,Q_0)$ with $\Psym\DAHA\Psym$
in Section~\ref{65} is
somewhat technical. It heavily uses the explicit relations given
in~\cite{17} for the algebras $AW(3,Q_0)$, $\AWQ$ and~$\DAHA$.
In Section~\ref{66} a similar isomorphism is proved between
$AW(3,Q_0)$ with two of its parameters $q$-shifted and
$\Psym^-\DAHA\Psym^-$, where $\Psym^-$ is the
antisymmetrizing idempotent in $\DAHA$.
In the f\/inal Section~\ref{67} it is shown as a corollary of these two
isomorphisms that $\AWQ$ is isomorphic with the centralizer of $T_1$ in
$\DAHA$. There it is also shown that the center of $AW(3,Q_0)$ is trivial,
with a proof in the same spirit as the proof of the faithfulness of
the basic representation of $AW(3,Q_0)$ given in \cite{17}.
Combination of the various results f\/inally gives as a corollary that
the center of $\DAHA$ is trivial.
\subsection*{Conventions}
Throughout assume that $q$ and $a$, $b$, $c$, $d$ are complex constants such that
\begin{gather}
q\ne0,\quad
q^m\ne1\ (m=1,2,\ldots),\quad
a,b,c,d\ne0,\quad
abcd\ne q^{-m}\ (m=0,1,2,\ldots).
\label{6}
\end{gather}
Let $e_1$, $e_2$, $e_3$, $e_4$ be the elementary symmetric polynomials in $a$, $b$, $c$, $d$:
\begin{gather}
e_1:=a+b+c+d,\qquad
e_2:=ab+ac+bc+ad+bd+cd,\nonumber\\
e_3:=abc+abd+acd+bcd,\qquad
e_4:=abcd.
\label{4}
\end{gather}
For Laurent polynomials $f$ in $z$ the $z$-dependence
will be written as $f[z]$.
Symmetric Laurent polynomials
$f[z]=\sum_{k=-n}^n c_k z^k$ (where $c_k=c_{-k}$) are related to ordinary
polynomials $f(x)$ in $x=\thalf(z+z^{-1})$ by
$f(\thalf(z+z^{-1}))=f[z]$.

\section{Summary of earlier results}

This section summarizes some results from \cite{17}, while Remarks \ref{75}
and \ref{73} about the duality anti-algebra isomorphism are new here.

\subsection[Askey-Wilson polynomials]{Askey--Wilson polynomials}

The {\em Askey--Wilson polynomials} are given by
\[
p_n\bigl(\thalf(z+z^{-1});a,b,c,d\mid q\bigr)
:=\frac{(ab,ac,ad;q)_n}{a^n}\,
\qhyp43{q^{-n},q^{n-1}abcd,az,az^{-1}}{ab,ac,ad}{q,q}
\]
(see \cite{2} for the def\/inition of $q$-shifted factorials $(a;q)_n$
and of $q$-hypergeometric series ${}_r\phi_s$).
These polynomials are symmetric in $a$, $b$, $c$, $d$.
We will work with the renormalized version which is {\em monic} as a
Laurent polynomial in $z$ (i.e., the coef\/f\/icient of $z^n$ equals 1):
\begin{gather}
P_n[z]=P_n[z;a,b,c,d\mid q]:=
\frac1{(abcdq^{n-1};q)_n}\,
p_n\bigl(\thalf(z+z^{-1});a,b,c,d\mid q\bigr).
\label{25}
\end{gather}
\par
The polynomials $P_n[z]$ are eigenfunctions of the operator $\Dsym$ acting
on the space $\Asym$
of symmetric Laurent polynomials $f[z]=f[z^{-1}]$:
\begin{gather}
(\Dsym f)[z]:=\frac{(1-az)(1-bz)(1-cz)(1-dz)}{(1-z^2)(1-qz^2)}\,
\bigl(f[qz]-f[z]\bigr)\nonumber\\
\phantom{(\Dsym f)[z]:=}{} +\frac{(a-z)(b-z)(c-z)(d-z)}{(1-z^2)(q-z^2)}\,
\bigl(f[q^{-1}z]-f[z]\bigr)
+(1+q^{-1}abcd)f[z].
\label{8}
\end{gather}
The eigenvalue equation is
\begin{gather}
\Dsym P_n=\la_nP_n,\qquad
\la_n:=q^{-n}+abcd q^{n-1}.
\label{22}
\end{gather}
Under condition \eqref{6} all eigenvalues in \eqref{22} are distinct.
\par
The three-term recurrence relation for the monic Askey--Wilson polynomials
is
\[
(z+z^{-1})P_n[z]=
P_{n+1}[z]+\be_n P_n[z]+\ga_n P_{n-1}[z]\qquad(n=0,1,2,\ldots),
\]
where $\be_n$ and $\ga_n$ are suitable constants and $P_{-1}[z]:=0$
(see \cite[(3.1.5)]{4}).

\subsection{Zhedanov's algebra}

{\em Zhedanov's algebra} $AW(3)$ (see \cite{1,21})
can in the $q$-case be described as an algebra with two generators
$K_0$, $K_1$ and with two relations
\begin{gather}
\label{1}
(q+q^{-1})K_1K_0K_1-K_1^2K_0-K_0K_1^2 =B\,K_1+ C_0\,K_0+D_0,\\
\label{2}
(q+q^{-1})K_0K_1K_0-K_0^2K_1-K_1K_0^2 =B\,K_0+C_1\,K_1+D_1.
\end{gather}
Here the {\em structure constants}
$B$, $C_0$, $C_1$, $D_0$, $D_1$ are f\/ixed complex constants.
\par
There is a {\em Casimir operator} $Q$ commuting with $K_0$, $K_1$:
\begin{gather}
Q=K_1K_0K_1K_0-(q^2+1+q^{-2})K_0K_1K_0K_1+(q+q^{-1})K_0^2K_1^2\nonumber\\
\phantom{Q=}{}+(q+q^{-1})(C_0K_0^2+C_1K_1^2)
+B\bigl((q+1+q^{-1})K_0K_1+K_1K_0\bigr)\nonumber\\
\phantom{Q=}{}+(q+1+q^{-1})(D_0K_0+D_1K_1).\label{3}
\end{gather}

Let the structure constants be expressed
in terms of $a$, $b$, $c$, $d$ by means of $e_1$, $e_2$, $e_3$, $e_4$
(see~\eqref{4}) as follows:
\begin{gather}
B :=(1-q^{-1})^2(e_3+qe_1),\nonumber\\
C_0 :=(q-q^{-1})^2,\nonumber\\
C_1 :=q^{-1}(q-q^{-1})^2 e_4,\label{16}\\
D_0 :=-q^{-3}(1-q)^2(1+q)(e_4+qe_2+q^2),\nonumber\\
D_1 :=-q^{-3}(1-q)^2(1+q)(e_1e_4+qe_3).\nonumber
\end{gather}
Then there is a representation (the {\em basic representation} or
{\em polynomial representation}) of the algebra
$AW(3)$ with structure constants \eqref{16}
on the space $\Asym$
of symmetric Laurent polynomials
as follows:
\begin{gather}
(K_0f)[z]:=(\Dsym f)[z],\qquad
(K_1f)[z]:=(Z+Z^{-1})f[z]:=(z+z^{-1})f[z],
\label{7}
\end{gather}
where $\Dsym$
is the operator \eqref{8} having the
{\em Askey--Wilson polynomials}
as eigenfunctions.
The Casimir operator $Q$ becomes constant in this representation:
\begin{gather}
(Qf)(z)=Q_0\,f(z),
\label{9}
\end{gather}
where
\begin{gather}
Q_0:=q^{-4}(1-q)^2\Bigl(q^4(e_4-e_2)+q^3(e_1^2-e_1e_3-2e_2)\nonumber\\
\phantom{Q_0:=}{} -q^2(e_2e_4+2e_4+e_2)
+q(e_3^2-2e_2e_4-e_1e_3)+e_4(1-e_2)\Bigr).
\label{10}
\end{gather}
(Note the slight error in this formula in \cite[(2.8)]{17}, version {\tt v3}.
It is corrected in {\tt v4}.)
\begin{Remark}\rm
\label{75}
Write $AW(3)=AW(3;K_0,K_1;a,b,c,d;q)$ in order to emphasize the dependence
of $AW(3)$ on the generators and the parameters (by \eqref{1},
\eqref{2}, \eqref{16}). There are several symmetries of this algebra.
First of all it is invariant under permutations of $a$, $b$, $c$, $d$.
The following one will be compatible with the duality of the double af\/f\/ine
Hecke algebra to be discussed below:

There is an anti-algebra isomorphism
\begin{gather}
AW(3;K_0,K_1;a,b,c,d;q)\to
AW\Biggl(3;a K_1,(q^{-1}abcd)^{-\half}K_0;\nonumber\\
\phantom{AW(3;K_0,K_1;a,b,c,d;q)\to}{} \frac1{(q^{-1}abcd)^\half},
\frac{ab}{(q^{-1}abcd)^\half},\frac{ac}{(q^{-1}abcd)^\half},
\frac{ad}{(q^{-1}abcd)^\half}\,;q\Biggr).\!\!\!
\label{70}
\end{gather}
See \cite[\S~2]{1}. Under this mapping $Q$ is sent to
$qa(bcd)^{-1}Q$ and $Q_0$ to $qa(bcd)^{-1}Q_0$.

Note that there is also a trivial anti-algebra isomorphism
\[
AW(3;K_0,K_1;a,b,c,d;q)\to AW(3;K_0,K_1;a,b,c,d;q).
\]
This sends $Q$ to $Q$ and $Q_0$ to $Q_0$.

There is also an algebra isomorphism
\[
AW(3;K_0,K_1;a,b,c,d;q)\to AW\left(3;\frac q{abcd}\,K_0,K_1;
a^{-1},b^{-1},c^{-1},d^{-1};q^{-1}\right).
\]
This sends $Q$ to $q^2(abcd)^{-2}Q$ and $Q_0$ to $q^2(abcd)^{-2}Q_0$.
\end{Remark}

Let $AW(3,Q_0)$ be the algebra generated by $K_0,K_1$ with relations
\eqref{1}, \eqref{2} and
\mbox{$Q=Q_0$},
assuming the structure constants \eqref{16}.
Then the basic representation of $AW(3)$ is also a representation of
$AW(3,Q_0)$.
We have the following theorem (see \cite[Theorem 2.2]{17}):
\begin{Theorem}
\label{12}
The elements
\[
K_0^n(K_1K_0)^lK_1^m\qquad(m,n=0,1,2,\ldots,\;l=0,1)
\]
form a basis of $AW(3,Q_0)$ and the representation
\eqref{7} of $AW(3,Q_0)$ is faithful.
\end{Theorem}
Note that the anti-algebra isomorphism \eqref{70} induces an anti-algebra
isomorphisms for $AW(3,Q_0)$.

\subsection[The double affine Hecke algebra of type
$(C_1^\vee,C_1)$]{The double af\/f\/ine Hecke algebra of type
$\boldsymbol{(C_1^\vee,C_1)}$}

One of the ways to describe the {\em double affine Hecke algebra} of type
$(C_1^\vee,C_1)$, denoted by $\DAHA$, is as follows
(see \cite[Proposition 5.2]{17}):
\begin{Definition}\rm
$\DAHA$ is the algebra generated by
$T_1$, $Y$, $Y^{-1}$, $Z$, $Z^{-1}$ with relations
$YY^{-1}=1=Y^{-1}Y$, $ZZ^{-1}=1=Z^{-1}Z$ and
\begin{gather}
T_1^2=-(ab+1)T_1-ab,\nonumber\\
T_1Z = Z^{-1}T_1+(ab+1)Z^{-1}-(a+b),\nonumber\\
T_1Z^{-1}=ZT_1-(ab+1)Z^{-1}+(a+b),\nonumber\\
T_1Y=q^{-1}abcd Y^{-1}T_1-(ab+1)Y+ab(1+q^{-1}cd),\nonumber\\
T_1Y^{-1}=q(abcd)^{-1}YT_1+q(abcd)^{-1}(1+ab)Y-q(cd)^{-1}(1+q^{-1}cd),\nonumber\\
YZ=qZY+(1+ab)cd\,Z^{-1}Y^{-1}T_1
-(a+b)cd\,Y^{-1}T_1
-(1+q^{-1}cd)Z^{-1}T_1\nonumber\\
\phantom{YZ=}{}
-(1-q)(1+ab)(1+q^{-1}cd)Z^{-1}
+(c+d)T_1
+(1-q)(a+b)(1+q^{-1}cd),\nonumber\\
YZ^{-1}=q^{-1}Z^{-1}Y
-q^{-2}(1+ab)cd\,Z^{-1}Y^{-1}T_1
+q^{-2}(a+b)cd\,Y^{-1}T_1\nonumber\\
\phantom{YZ^{-1}=}{}
+q^{-1}(1+q^{-1}cd)Z^{-1}T_1
-q^{-1}(c+d)T_1,\nonumber\\
Y^{-1}Z=q^{-1}ZY^{-1}-q(ab)^{-1}(1+ab)Z^{-1}Y^{-1}T_1
+(ab)^{-1}(a+b)Y^{-1}T_1\nonumber\\
\phantom{Y^{-1}Z=}{}
+q(abcd)^{-1}(1+q^{-1}cd)Z^{-1}T_1
+q(abcd)^{-1}(1-q)(1+ab)(1+q^{-1}cd)Z^{-1}\nonumber\\
\phantom{Y^{-1}Z=}{}
-(abcd)^{-1}(c+d)T_1
-(abcd)^{-1}(1-q)(1+ab)(c+d),\nonumber\\
Y^{-1}Z^{-1}=qZ^{-1}Y^{-1}+q(ab)^{-1}(1+ab)Z^{-1}Y^{-1}T_1
-(ab)^{-1}(a+b)Y^{-1}T_1\nonumber\\
\phantom{Y^{-1}Z^{-1}=}{}
-q^2(abcd)^{-1}(1+q^{-1}cd)Z^{-1}T_1
+q(abcd)^{-1}(c+d)T_1.
\label{39}
\end{gather}
\end{Definition}

By adding the relations for $TZ$ and $TZ^{-1}$ and by combining the relations
for $TY$ and $TY^{-1}$ we see that
\[
T_1(Z+Z^{-1})=(Z+Z^{-1})T_1,\qquad
T_1(Y+q^{-1}abcdY^{-1})=(Y+q^{-1}abcdY^{-1})T_1.
\]
\par
For the following theorem see Sahi \cite{9} in the general rank case.
In \cite[Theorem 5.3]{17} it is proved for the rank one case only.

\begin{Theorem}
A basis of $\DAHA$ is provided by the elements $Z^mY^nT_1^i$, where
$m,n\in\ZZ$, $i=0,1$.
\end{Theorem}

The {\em basic} or {\em polynomial} representation of $\DAHA$
is a representation
of $\DAHA$ on the space $\FSA$ of Laurent polynomials $f[z]$ in $z$ such that
$(Zf)[z]:=z\,f[z]$ and $T_1$ and $Y$ act as $q$-dif\/ference-ref\/lection
operators given by \cite[(3.11), (3.13)]{17}. This representation is
faithful. If $ab\ne1$ then $T_1$ acting on~$\FSA$ has eigenspaces $\Asym$
(for eigenvalue $-ab$) and $\Asym^-$ (for eigenvalue $-1$), and $\FSA$ is
the direct sum of $\Asym$ and $\Asym^-$.

In the basic representation of $\DAHA$ the operator $Y+q^{-1}abcdY^{-1}$
has eigenvalues $\la_n$ (see \eqref{22}),
of multiplicity 2 for $n=1,2,\ldots$
and of multiplicity 1 for $n=0$. If $n\ge1$ and $ab\ne1$
then the eigenspace of $\la_n$
splits as a one-dimensional part in $\Asym$ spanned by the
Askey--Wilson polynomial $P_n[z]$ given by \eqref{25} and a
one-dimensional part in $\Asym^-$ spanned by
\begin{gather}
Q_n[z]:=a^{-1}b^{-1}z^{-1}(1-az)(1-bz)\,P_{n-1}[z;qa,qb,c,d\mid q].
\label{72}
\end{gather}
\par
In \cite[\S~6]{17} the following
algebra $\AWQ$ was def\/ined, which is a central extension
of $AW(3,Q_0)$:

\begin{Definition}\rm
\label{41}
The algebra $\AWQ$ is generated by
$K_0$, $K_1$, $T_1$ such that
$(T_1+ab)(T_1+1)=0$, $T_1$ commutes with $K_0$, $K_1$, and with further relations
\begin{gather}
 (q+q^{-1})K_1K_0K_1-K_1^2K_0-K_0K_1^2
=B\,K_1+ C_0\,K_0+D_0\nonumber\\
\qquad{}+E\,K_1(T_1+ab)+F_0(T_1+ab),
\label{34}\\
(q+q^{-1})K_0K_1K_0-K_0^2K_1-K_1K_0^2
=B\,K_0+C_1\,K_1+D_1\nonumber\\
\qquad{}+E\,K_0(T_1+ab)+F_1(T_1+ab),
\label{35}\\
Q_0=(K_1K_0)^2-(q^2+1+q^{-2})K_0(K_1K_0)K_1
+(q+q^{-1})K_0^2K_1^2\nonumber\\
\qquad{}+(q+q^{-1})(C_0K_0^2+C_1K_1^2)
+\bigl(B+E(T_1+ab)\bigr)
\bigl((q+1+q^{-1})K_0K_1+K_1K_0\bigr)\nonumber\\
\qquad{}+(q+1+q^{-1})\bigl(D_0+F_0(T_1+ab)\bigr)K_0
+(q+1+q^{-1})\bigl(D_1+F_1(T_1+ab)\bigr)K_1\nonumber\\
\qquad{}+
G(T_1+ab).
\label{36}
\end{gather}
Here the structure constants are given by \eqref{16} together with
\begin{gather*}
E :=-q^{-2}(1-q)^3(c+d),\\
F_0 :=q^{-3}(1-q)^3(1+q)(cd+q),\\
F_1 :=q^{-3}(1-q)^3(1+q)(a+b)cd,\\
G :=-q^{-4}(1-q)^3\Bigl((a+b)(c+d)\bigl(cd(q^2+1)+q\bigr)
-q(ab+1)\bigl((c^2+d^2)(q+1)-cd\bigr)\\
\phantom{G :=}{} +(cd+e_4)(q^2+1)+(e_2+e_4-ab)q^3\Bigr),
\end{gather*}
and $Q_0$ is given by \eqref{10}.
\end{Definition}

Then it was proved in \cite[Theorem 6.2, Corollary 6.3]{17}:
\begin{Theorem}
\label{42}
$\AWQ$ has a basis consisting of
\[
K_0^n(K_1K_0)^iK_1^mT_1^j\qquad(m,n=0,1,2,\ldots,\ i,j=0,1).
\]
There is a unique algebra isomorphism from $\AWQ$ into $\DAHA$ such that
$K_0\mapsto Y+q^{-1}abcdY^{-1}$,  $K_1\mapsto Z+Z^{-1}$,
$T_1\mapsto T_1$. The elements in the image commute with $T_1$.
\end{Theorem}

\begin{Remark}\rm
\label{73}
Write $\DAHA=\tilde\goH(Y,Z,T_1;a,b,c,d;q)$
in order to emphasize the dependence
of $\DAHA$ on the generators and the parameters. Then there is an
anti-algebra isomorphism
\begin{gather*}
\tilde\goH(Y,Z,T_1;a,b,c,d;q)\to
\tilde\goH\Biggl(aZ^{-1},(q^{-1}abcd)^{-\half}Y^{-1},T_1;\\
\phantom{\tilde\goH(Y,Z,T_1;a,b,c,d;q)\to}{}
\frac1{(q^{-1}abcd)^\half},
\frac{ab}{(q^{-1}abcd)^\half},\frac{ac}{(q^{-1}abcd)^\half},
\frac{ad}{(q^{-1}abcd)^\half}\,;q\Biggr),
\end{gather*}
see \cite[Proposition 8.5(i)]{11}.
Also write $\AWQ=\widetilde{AW}(3,Q_0;K_0,K_1,T_1;a,b,c,d;q)$. Then, by a
slight adaptation of \eqref{70}, there is an anti-algebra isomorphism
\begin{gather*}
\widetilde{AW}(3,Q_0;K_0,K_1,T_1;a,b,c,d;q)\to
\widetilde{AW}\Biggl(3;qa(bcd)^{-1}Q_0;a K_1,(q^{-1}abcd)^{-\half}K_0;\\
\phantom{\widetilde{AW}(3,Q_0;K_0,K_1,T_1;a,b,c,d;q)\to}{} \frac1{(q^{-1}abcd)^\half},
\frac{ab}{(q^{-1}abcd)^\half},\frac{ac}{(q^{-1}abcd)^\half},
\frac{ad}{(q^{-1}abcd)^\half};q\Biggr).
\end{gather*}
The two anti-algebra isomorphisms are compatible under the algebra embedding
of $\AWQ$ into $\DAHA$ given in Theorem \ref{42}.
\end{Remark}

\section{The spherical subalgebra}
\label{65}
{}From now on assume $ab\ne1$.
In $\DAHA$ put
\[
\Psym:=(1-ab)^{-1}(T_1+1).
\]
Then
\[
\Psym^2=\Psym.
\]
In the basic representation of $\DAHA$ we have for $f\in\FSA$:
\[
\Psym f=
\begin{cases}
f&\mbox{if \ $T_1f=-abf$,}\\
0&\mbox{if \ $T_1f=-f$.}
\end{cases}
\]
$\Psym$ projects $\FSA$ onto $\Asym$.
Def\/ine the linear map $S\colon\DAHA\to\DAHA$ by
\[
S(U):=\Psym U\Psym\qquad(U\in\DAHA).
\]
Then
\[
S(U)\,S(V)=S(U\Psym V)\qquad(U,V\in\DAHA).
\]
Hence the image $S(\DAHA)$
is a subalgebra of $\DAHA$. We call it the {\em spherical subalgebra} of
$\DAHA$.
\par
For $U\in\DAHA$ we have in the basic representation:
\[
S(U)\,f=\Psym U\Psym f=
\begin{cases}
\Psym U\,f&\mbox{if \ $T_1f=-abf$,}\\
0&\mbox{if \ $T_1f=-f$.}
\end{cases}
\]
Hence, for the basic representation of $\DAHA$ restricted to $S(\DAHA)$,
$\Asym$ is an invariant subspace. This representation of $S(\DAHA)$ on
$\Asym$ is faithful.
Indeed, if $S(U)\,f=0$ for all $f\in\Asym$ then $S(U)\,f=0$ for all $f\in\FSA$,
so $S(U)=0$ by the faithfulness of the basic representation of $\DAHA$
on $\FSA$.
\par
$Z_{\DAHA}(T_1)$,
the centralizer of $T_1$ in $\DAHA$, is a subalgebra of $\DAHA$. It has
$\Psym$ as a central element. Hence
\begin{gather}
S(U)=U\Psym\qquad{\rm and}\qquad S(UV)=S(U)\,S(V)
\qquad(U,V\in Z_{\DAHA}(T_1)).
\label{33}
\end{gather}
So $S$ restricted to $Z_{\DAHA}(T_1)$
is an algebra homomorphism.
\par
The algebra $\AWQ$ was def\/ined by Def\/inition \ref{41}.
By Theorem \ref{42}
there is an algeb\-ra isomorphism from $\AWQ$ into $\DAHA$ such that
$K_0\mapsto Y+q^{-1}abcdY^{-1}$,  $K_1\mapsto Z+Z^{-1}$,
\mbox{$T_1\mapsto T_1$}.
This isomorphism embeds $\AWQ$ into $Z_{\DAHA}(T_1)$.
So we may consider $\AWQ$ as a~subalgebra of $Z_{\DAHA}(T_1)$ and
\eqref{33} will hold for $U,V\in \AWQ$.
\par
By \eqref{1}, \eqref{2}, \eqref{9} and \eqref{3},
the algebra $AW(3,Q_0)$ can be presented by the same genera\-tors and relations
as for $\AWQ$ but with additional relation $T_1=-ab$.
\par
By Theorem \ref{42} $\AWQ$ has a basis consisting of the elements
\[
K_0^n(K_1K_0)^iK_1^m(T_1+1),\qquad K_0^n(K_1K_0)^iK_1^m(T_1+ab)
\qquad
(m,n=0,1,2,\ldots,\  \ i=0,1).
\]
Hence $S(\AWQ)$ has a basis consisting of
\[
(1-ab)^{-1}K_0^n(K_1K_0)^iK_1^m(T_1+1)
\qquad(m,n=0,1,2,\ldots,\ \ i=0,1).
\]
By Theorem \ref{12} $AW(3,Q_0)$ has a basis consisting of
\begin{gather}
K_0^n(K_1K_0)^iK_1^m
\qquad(m,n=0,1,2,\ldots,\ \ i=0,1).
\label{57}
\end{gather}
Hence the map which sends a basis element $K_0^n(K_1K_0)^iK_1^m$
of $AW(3,Q_0)$ to a basis element $(1-ab)^{-1}K_0^n(K_1K_0)^iK_1^m(T_1+1)$
of $S(\AWQ)$ extends linearly to a linear bijection from
$AW(3,Q_0)$ onto $S(\AWQ)$. In fact, this map remains well-def\/ined if
we write it as
\begin{gather}
U\mapsto (1-ab)^{-1}\wt U(T_1+1),
\label{32}
\end{gather}
where $U\mapsto \wt U$ sends words $U$ in $AW(3,Q_0)$ to
corresponding words $\wt U$ in $\AWQ$. Moreover, the map \eqref{32}
can then be seen to be an algebra homomorphism.
Indeed, consider the linear map $U\mapsto\wt U$ as a map to $\AWQ$
from the free algebra
generated by $K_0$, $K_1$, $T_1$ with $T_1$ central such that it sends a word
involving $K_0$, $K_1$, $T_1$ to the same word in $\AWQ$.
This map is an algebra homomorphism.
Composing it with $S$ yields the map \eqref{32} which is
again an algebra homomorphism.
Now we have to check
that $R$ is sent to zero by the map \eqref{32}
if $R=0$ is a~relation for $AW(3,Q_0)$.
This is clearly the case for $R:=T_1+ab$, since $(T_1+ab)(T_1+1)=0$.
It is also clear for the other
relations $R=0$ in $AW(3,Q_0)$ since these can be taken as the relations
\eqref{34}--\eqref{36}, which are also relations for $\AWQ$.
So we have shown:
\begin{Proposition}
The map \eqref{32}, where $U\mapsto\wt U$
sends words $U$ involving $K_0,K_1$ in $AW(3,Q_0)$
to the same words $\wt U$ in $\AWQ$, is a well defined algebra isomorphism
from $AW(3,Q_0)$ onto $S(\AWQ)$.
\end{Proposition}
\begin{Theorem}
\label{37}
$S(\DAHA)=S(\AWQ)$, so the spherical subalgebra
$S(\DAHA)$ is isomorphic to the algebra $AW(3,Q_0)$ by the
map \eqref{32} sending $AW(3,Q_0)$ to $S(\DAHA)$.
\end{Theorem}
For the proof note f\/irst that
$\DAHA$ has a basis consisting of the elements
\begin{gather}
Z^mY^n(T_1+1),\qquad Z^mY^n(T_1+ab)\qquad(m,n\in\ZZ).
\label{60}
\end{gather}
Hence
$S(\DAHA)$ is spanned by the elements
\[
(T_1+1)Z^mY^n(T_1+1)\qquad(m,n\in\ZZ).
\]
\begin{Definition}\rm
\label{61}
Let $m,n\in\ZZ$.
For an element in $\DAHA$ which is a linear combination of basis elements
$Z^kY^l$ we say that
\[
\sum_{k,l\in\ZZ}c_{k,l} Z^kY^l=o(Z^mY^n)
\]
if $c_{k,l}\ne0$
implies $|k|\le |m|$, $|l|\le|n|$, $(|k|,|l|)\ne(|m|,|n|)$.
\end{Definition}
Theorem \ref{37} will follow by induction with respect to $|m|+|n|$
from the following lemma:
\begin{Lemma}
\label{62}
Let $m,n\in\ZZ$. Then
\[
(T_1+1)Z^mY^n(T_1+1)\in
(T_1+1)\Bigl(\AWQ+o(Z^mY^n)\Bigr)(T_1+1).
\]
\end{Lemma}
\begin{proof}
The procedure will be as follows:
\begin{enumerate}\itemsep=0pt
\item
Write $(T_1+1)Z^mY^n(T_1+1)$ as a linear combination of
\begin{gather}
Z^{|m|}Y^{|n|}(T_1+1),\qquad
Z^{-|m|}Y^{|n|}(T_1+1),\qquad
Z^{|m|}Y^{-|n|}(T_1+1),\nonumber\\
Z^{-|m|}Y^{-|n|}(T_1+1)
\qquad
\pmod{o(Z^{|m|}Y^{|n|})(T_1+1)}.
\label{38}
\end{gather}
This is done by induction, starting with the $\DAHA$ relations for
$T_1Z$, $T_1Z^{-1}$, $T_1Y$, $T_1Y^{-1}$.
\item
Also write $K_1^mK_0^n(T_1+1)$ and
$K_1^{m-1}K_0K_1K_0^{n-1}(T_1+1)$
($m,n=0,1,\ldots$) as a linear combination of \eqref{38}.
\item
These latter linear combinations turn out to
span the linear combinations obtained for
$(T_1+1)Z^mY^n(T_1+1)$.
\end{enumerate}
{\bf Step 1.}
We get the following expressions in terms of the elements \eqref{38}
(here $m,n\in\{1,2,\ldots\}$)
\begin{gather}
(T_1+1)Z^m(T_1+1) =\Bigl(Z^m+Z^{-m}+o(Z^m)\Bigr)(T_1+1),
\label{44}\\
(T_1+1)Z^{-m}(T_1+1) =-ab\Bigl(Z^m+Z^{-m}+o(Z^m)\Bigr)(T_1+1),
\label{45}\\
(T_1+1)Y^n(T_1+1) =-ab\Bigl(Y^n+(q^{-1}abcd)^nY^{-n}+o(Y^n)\Bigr)(T_1+1),
\label{47}\\
(T_1+1)Y^{-n}(T_1+1) =\Bigl((q^{-1}abcd)^{-n}Y^n+Y^{-n}+o(Y^n)\Bigr)(T_1+1),
\label{48}\\
(T_1+1)Z^mY^n(T_1+1) =\Bigl(Z^mY^n-ab(q^{-1}abcd)^nZ^{-m}Y^{-n}
+o(Z^mY^n)\Bigr)(T_1+1),
\label{49}\\
(T_1+1)Z^{-m}Y^n(T_1+1) =\Bigl(-(ab+1)Z^mY^n-ab(q^{-1}abcd)^nZ^mY^{-n}
-ab Z^{-m}Y^n
\nonu\\
 \phantom{(T_1+1)Z^{-m}Y^n(T_1+1) =}{}
+o(Z^mY^n)\Bigr)(T_1+1),
\label{50}\\
(T_1+1)Z^mY^{-n}(T_1+1) =\Bigl(Z^mY^{-n}+(q^{-1}abcd)^{-n}Z^{-m}Y^n
+(1+ab)Z^{-m}Y^{-n}
\nonu\\
 \phantom{(T_1+1)Z^mY^{-n}(T_1+1) =}{}
+o(Z^mY^n)\Bigr)(T_1+1),
\label{51}\\
(T_1+1)Z^{-m}Y^{-n}(T_1+1) =\Bigl((q^{-1}abcd)^{-n}Z^mY^n-abZ^{-m}Y^{-n}\nonumber\\
\phantom{(T_1+1)Z^{-m}Y^{-n}(T_1+1) =}{}
+o(Z^mY^n)\Bigr)(T_1+1).
\label{52}
\end{gather}
{\bf Step 2.}
We get the following expressions in terms of the elements \eqref{38}
(here $m,n\in\{1,2,\ldots\}$)
\begin{gather}
K_1^m(T_1+1) =\Bigl(Z^m+Z^{-m}+o(Z^m)\Bigr)(T_1+1),
\label{53}\\
K_0^n(T_1+1) =\Bigl(Y^n+(q^{-1}abcd)^nY^{-n}+o(Y^n)\Bigr)(T_1+1),
\label{54}\\
K_1^mK_0^n(T_1+1) =\Bigl(Z^mY^n+Z^{-m}Y^n
+(q^{-1}abcd)^n\bigl(Z^m Y^{-n}+Z^{-m}Y^{-n}\bigr)
\nonu\\
\phantom{K_1^mK_0^n(T_1+1) =}{}
+o(Z^mY^n)\Bigr)(T_1+1),
\label{55}\\
K_1^{m-1}K_0K_1K_0^{n-1}(T_1+1) =\Bigl(qZ^mY^n+q^{-1}Z^{-m}Y^n
+q^{-1}(q^{-1}abcd)^nZ^m Y^{-n}
\nonu\\
 \qquad{}
+q^{-1}(q^{-1}abcd)^n(1+ab-q^2ab)Z^{-m}Y^{-n}
+o(Z^mY^n)\Bigr)(T_1+1).
\label{56}
\end{gather}
{\bf Step 3.}
The only cases which may not be immediately clear are for
$(T_1+1)Z^{\pm m}Y^{\pm n}(T_1+1)$ ($m,n\in\{1,2,\ldots\}$).
Then we have by comparing the identities in Step 1 and Step 2:
\begin{gather*}
(T_1+1)Z^mY^n(T_1+1) =
\frac1{(1-q^2)}\Bigl(K_1^{m-1}(K_1K_0-qK_0K_1)K_0^{n-1}+o(Z^mY^n)\Bigr)
(T_1+1),\\
(T_1+1)Z^{-m}Y^n(T_1+1) =
\frac q{(1-q^2)}\Bigl(K_1^{m-1}(-q^{-1}\bigl(1+ab-q^2ab)K_1K_0+
K_0K_1\bigr)K_0^{n-1}\\
 \phantom{(T_1+1)Z^{-m}Y^n(T_1+1) =}{}
+o(Z^mY^n)\Bigr)(T_1+1),\\
(T_1+1)Z^mY^{-n}(T_1+1) =
\frac q{(1-q^2)(q^{-1}abcd)^n}\Bigl(K_1^{m-1}(-qK_1K_0+K_0K_1\bigr)K_0^{n-1}\\
 \phantom{(T_1+1)Z^mY^{-n}(T_1+1) =}{}
+o(Z^mY^n)\Bigr)(T_1+1),\\
(T_1+1)Z^{-m}Y^{-n}(T_1+1) =
\frac 1{(1-q^2)(q^{-1}abcd)^n}\Bigl(K_1^{m-1}(K_1K_0-qK_0K_1\bigr)K_0^{n-1}\\
 \phantom{(T_1+1)Z^{-m}Y^{-n}(T_1+1) =}{}
+o(Z^mY^n)\Bigr)(T_1+1).
\end{gather*}
{\bf Proofs for Step 1.}
We will use repeatedly that $(T_1+1)^2=(1-ab)(T_1+1)$.
{}From the relation for $T_1Z$ in \eqref{39} we obtain
\begin{gather}
(T_1+1)Z=Z^{-1}(T_1+1)+Z+abZ^{-1}-(a+b),
\label{46}\\
(T_1+1)Z(T_1+1)=\bigl(Z+Z^{-1}-(a+b)\bigr)(T_1+1)=
\bigl(Z+Z^{-1}+o(Z)\bigr)(T_1+1),
\nonu
\end{gather}
i.e., \eqref{44} for $m=1$.
Now we will prove \eqref{44} by induction. Suppose it holds for some positive
integer $m$, then we will prove the identity with $m$ replaced by $m+1$.
By successively substitu\-ting~\eqref{46} and the induction hypothesis we have:
\begin{gather*}
(T_1+1)Z^{m+1}(T_1+1)=
Z^{-1}(T_1+1)Z^m(T_1+1)+
\bigl(Z^{m+1}+o(Z^{m+1})\bigr)(T_1+1)\\
\phantom{(T_1+1)Z^{m+1}(T_1+1)}{}=Z^{-1}\bigl(Z^m+Z^{-m}+o(Z^m)\bigr)(T_1+1)+
\bigl(Z^{m+1}+o(Z^{m+1})\bigr)(T_1+1)\\
\phantom{(T_1+1)Z^{m+1}(T_1+1)}{}
=\bigl(Z^{m+1}+Z^{-m-1}+o(Z^{m+1})\bigr)(T_1+1).
\end{gather*}

Similarly, from the relation for $T_1Z^{-1}$ in \eqref{39} we prove
\eqref{45} by induction:
\begin{gather*}
(T_1+1)Z^{-1}=Z(T_1+1)-Z-abZ^{-1}+(a+b),\quad\mbox{hence}\\
(T_1+1)Z^{-1}(T_1+1)=\bigl(-ab(Z+Z^{-1})+(a+b)\bigr)(T_1+1)\\
\phantom{(T_1+1)Z^{-1}(T_1+1)}{}
=-ab\bigl(Z+Z^{-1}+o(Z)\bigr)(T_1+1),\\
(T_1+1)Z^{-m-1}(T_1+1)=
Z(T_1+1)Z^{-m}(T_1+1)-ab
\bigl(Z^{-m-1}+o(Z^{m+1})\bigr)(T_1+1)\\
\qquad{}=-abZ\bigl(Z^m+Z^{-m}+o(Z^m)\bigr)(T_1+1)-ab
\bigl(Z^{-m-1}+o(Z^{m+1})\bigr)(T_1+1)\\
\qquad=-ab\bigl(Z^{m+1}+Z^{-m-1}+o(Z^{m+1})\bigr)(T_1+1).
\end{gather*}
The proofs of \eqref{47}, \eqref{48} are similar.

The proof of \eqref{49} is for f\/ixed $n$ by induction with respect to $m$.
First we prove the case $m=1$. By \eqref{46} and \eqref{47} we obtain:
\begin{gather*}
(T_1+1)ZY^n(T_1+1)
=\Bigl(Z^{-1}(T_1+1)Y^n+ZY^n+abZ^{-1}Y^n+o(ZY^n)\Bigr)(T_1+1)\\
\phantom{(T_1+1)ZY^n(T_1+1)}{}=\Bigl(ZY^n-ab(q^{-1}abcd)^nZ^{-1}Y^{-n}
+o(ZY^n)\Bigr)(T_1+1).
\end{gather*}
Now suppose \eqref{49} holds for some positive
integer $m$, then we will prove the identity with $m$ replaced by $m+1$.
By \eqref{46} and the induction hypothesis we obtain:
\begin{gather*}
(T_1+1)Z^{m+1}Y^n(T_1+1)
=\Bigl(Z^{-1}(T_1+1)Z^mY^n+Z^{m+1}Y^n+o(Z^{m+1}Y^n)\Bigr)(T_1+1)\\
\phantom{(T_1+1)Z^{m+1}Y^n(T_1+1)}{}=\Bigl(Z^{m+1}Y^n\!-ab(q^{-1}abcd)^nZ^{-m-1}Y^{-n}\!
+o(Z^{m+1}Y^n)\Bigr)(T_1+1).
\end{gather*}
The proofs of \eqref{50}, \eqref{51}, \eqref{52} are similar.

\medskip

\noindent
{\bf Proofs for Step 2.}
Formulas \eqref{53}, \eqref{54}, \eqref{55} immediately follow by the
substitutions $K_0=Y+q^{-1}abcdY^{-1}$, $K_1=Z+Z^{-1}$.
As for \eqref{56} we f\/irst verify it for $m=n=1$ by
writing the \LHS\ as $(1-ab)^{-1}(T_1+1)(Y+q^{-1}abcdY^{-1})(Z+Z^{-1})(T_1+1)$
and by using the last four relations in \eqref{39}. Then we obtain for
the case of general $m$, $n$:
\begin{gather*}
K_1^{m-1}K_0K_1K_0^{n-1}(T_1+1)
=(1-ab)^{-1}K_1^{m-1}K_0K_1(T_1+1)K_0^{n-1}(T_1+1)\\
\qquad{}=(Z+Z^{-1})^{m-1}\Bigl(qZY+q^{-1}Z^{-1}Y
+q^{-1}(q^{-1}abcd)Z Y^{-1}\\
\qquad\qquad{}+q^{-1}(q^{-1}abcd)(1+ab-q^2ab)Z^{-1}Y^{-1}
+o(ZY)\Bigr)(Y+q^{-1}abcdY^{-1})^{n-1}(T_1+1)\\
\qquad{}=\Bigl(qZ^mY^n+q^{-1}Z^{-m}Y^n
+q^{-1}(q^{-1}abcd)^nZ^m Y^{-n}\\
\qquad\qquad{}
+q^{-1}(q^{-1}abcd)^n(1+ab-q^2ab)Z^{-m}Y^{-n}
+o(Z^mY^n)\Bigr)(T_1+1).\tag*{\qed}
\end{gather*}
\renewcommand{\qed}{}
\end{proof}

\section{The antispherical subalgebra}
\label{66}
In $\DAHA$ put
\[
\Psym^-:=(ab-1)^{-1}(T_1+ab).
\]
Then
\[
(\Psym^-)^2=\Psym^-.
\]
In the basic representation of $\DAHA$ we have for $f\in\FSA$:
\[
\Psym^- f=
\begin{cases}
f&\mbox{if \ $T_1f=-f$,}\\
0&\mbox{if \ $T_1f=-abf$.}
\end{cases}
\]
Let $\Asym^-$ denote the eigenspace of $T_1$ acting on $\FSA$
for eigenvalue $-1$. Then
$\Psym^-$ projects $\FSA$ onto~$\Asym^-$.
Def\/ine the linear map $S^-\colon\DAHA\to\DAHA$ by
\[
S^-(U):=\Psym^- U\Psym^-\qquad(U\in\DAHA).
\]
Then
\[
S^-(U)\,S^-(V)=S^-(U\Psym^- V)\qquad(U,V\in\DAHA).
\]
Hence the image $S^-(\DAHA)$
is a subalgebra of $\DAHA$. We call it the {\em antispherical subalgebra} of
$\DAHA$.
\par
For $U\in\DAHA$ we have in the basic representation:
\[
S^-(U)\,f=\Psym^- U\Psym^- f=
\begin{cases}
\Psym^- U\,f&\mbox{if \ $T_1f=-f$,}\\
0&\mbox{if \ $T_1f=-abf$.}
\end{cases}
\]
Hence, for the basic representation of $\DAHA$ restricted to $S^-(\DAHA)$,
$\Asym^-$ is an invariant subspace.
\par
Recall the algebra isomorphism from $\AWQ$ into $Z_{\DAHA}(T_1)$ given
by Theorem \ref{42}.
Since $S^-(U)\,S^-(V)=S^-(UV)$ for $U,V\in Z_{\DAHA}(T_1)$, we see that
$\AWQ$, considered as a subalgabra of $\DAHA$ and hence of
$Z_{\DAHA}(T_1)$, is mapped by $S^-$
onto a subalgebra $S^-(\AWQ)$ of $S^-(\DAHA)$.
{}From the basis \eqref{57} of $\AWQ$ we see that $S^-(\AWQ)$ has basis
\[
(ab-1)^{-1}K_0^n(K_1K_0)^iK_1^m(T_1+ab)
\qquad(m,n=0,1,2,\ldots,\  i=0,1).
\]
Compare this basis with the basis \eqref{57} of $AW(3,Q_0)$.
Thus the map which sends a basis element $K_0^n(K_1K_0)^iK_1^m$
of $AW(3,Q_0)$ to a basis element
$q^{n+i}(ab-1)^{-1}K_0^n(K_1K_0)^iK_1^m(T_1+ab)$
of $S^-(\AWQ)$ extends linearly to a linear bijection from
$AW(3,Q_0)$ onto $S^-(\AWQ)$. In fact this map extends to an
algebra isomorphism:
\begin{Proposition}
\label{69}
Let $AW(3,Q_0;qa,qb,c,d)$ be $AW(3,Q_0)$ with $a$, $b$ replaced by
$qa$, $qb$, respectively. Then there is a well-defined
algebra isomorphism from
$AW(3,Q_0;qa,qb,c,d)$ onto $S^-(\AWQ)$ given by the map
\begin{gather}
U\mapsto (ab-1)^{-1}\wt U(T_1+ab),
\label{58}
\end{gather}
where $U\mapsto \wt U$ sends words $U$ in $AW(3,Q_0;qa,qb,c,d)$
involving $K_0$, $K_1$ to
corresponding words~$\wt U$ in $\AWQ$ involving $qK_0$, $K_1$.
\end{Proposition}

\begin{proof}
Consider the linear map $U\mapsto\wt U$ as a map to $\AWQ$
from the free algebra
generated by $K_0$, $K_1$, $T_1$ with $T_1$ central such that it sends a word
involving $K_0$, $K_1$, $T_1$ to the corresponding  word
involving $qK_0$, $K_1$, $T_1$
in $\AWQ$. This map is an algebra homomorphism.
Composing it with $S^-$ yields the map \eqref{58} which is
again an algebra homomorphism.
Now we have to check
that $R$ is sent to zero by the map \eqref{58}
if $R=0$ is a relation for $AW(3,Q_0;qa,qb,c,d)$.
This is clearly the case for $R:=T_1+1$, since $(T_1+1)(T_1+ab)=0$.
To see this for the other relations,
rewrite relations \eqref{34}--\eqref{36} for $\AWQ$ as:
\begin{gather*}
(q+q^{-1})K_1K_0K_1-K_1^2K_0-K_0K_1^2
-\bigl(B+(ab-1)E\bigr)K_1- C_0\,K_0-\bigl(D_0+(ab-1)F_0\bigr)
\\
\qquad{}
-(E\,K_1+F_0)(T_1+1)=0,
\\
(q+q^{-1})K_0K_1K_0-K_0^2K_1-K_1K_0^2
-\bigl(B+(ab-1)E\bigr)K_0-C_1\,K_1-\bigl(D_1+(ab-1)F_1\bigr)
\\
\qquad{}
-(E\,K_0+F_1)(T_1+1)=0,
\\
K_1K_0K_1K_0-(q^2+1+q^{-2})K_0K_1K_0K_1
+(q+q^{-1})K_0^2K_1^2-((q+q^{-1})(C_0K_0^2+C_1K_1^2)
\\
\qquad{}+\bigl(B+(ab-1)E\bigr)
\bigl((q+1+q^{-1})K_0K_1\!+K_1K_0\bigr)\!
+(q+1+q^{-1})\bigl(D_0\!+(ab-1)F_0\bigr)K_0\!
\\
\qquad{}+(q+1+q^{-1})\bigl(D_1+(ab-1)F_1\bigr)K_1+(ab-1)G-Q_0+
(E\,K_1K_0+G)(T_1+1)
\\
\quad+(q+1+q^{-1})(E\,K_0K_1
+F_0\,K_0+F_1\,K_1)(T_1+1)=0.
\end{gather*}
On multiplication with $T_1+ab$ we see that the identities
$R_i(T_1+ab)=0$ ($i=1,2,3$)
must be valid in $S^-(\AWQ)$, where
\begin{gather*}
R_1 :=(q+q^{-1})K_1K_0K_1-K_1^2K_0-K_0K_1^2
-\bigl(B+(ab-1)E\bigr)K_1\\
\phantom{R_1 :=}{} - C_0\,K_0-D_0-(ab-1)F_0,
\\
R_2 :=(q+q^{-1})K_0K_1K_0-K_0^2K_1-K_1K_0^2
-(B+(ab-1)E\bigr)K_0\\
\phantom{R_2 :=}{}-C_1\,K_1-D_1-(ab-1)F_1,
\\
R_3 :=K_1K_0K_1K_0-(q^2+1+q^{-2})K_0K_1K_0K_1
+(q+q^{-1})K_0^2K_1^2\\
\phantom{R_3 :=}{} -((q+q^{-1})(C_0K_0^2+C_1K_1^2)
+\bigl(B+(ab-1)E\bigr)
\bigl((q+1+q^{-1})K_0K_1+K_1K_0\bigr)\\
\phantom{R_3 :=}{}
+(q+1+q^{-1})\bigl(D_0+(ab-1)F_0\bigr)K_0
+(q+1+q^{-1})\bigl(D_1+(ab-1)F_1\bigr)K_1\\
\phantom{R_3 :=}{}
+(ab-1)G-Q_0.
\end{gather*}
Now consider relations \eqref{1}, \eqref{2} en \eqref{9} for
$AW(3,Q_0;qa,qb,c,d)$ (so with $a$, $b$ in the structure constants replaced by
$qa$, $qb$).
These can be written in the form $U_i=0$ ($i=1,2,3$) by bringing everything
to the \LHS\ in the relations. Now consider $U_i$ as elements of the free
algebra generated by $K_0$, $K_1$. For the images under the map \eqref{58}
we then obtain the following elements of $S^-(\AWQ)$:{\samepage
\[
\wt U_1=(1-ab)^{-1} qR_1,\qquad
\wt U_2=(1-ab)^{-1} q^2R_2,\qquad
\wt U_3=(1-ab)^{-1} q^2R_3,
\]
which are all zero.}
\end{proof}

\begin{Theorem}
\label{59}
$S^-(\DAHA)=S(\AWQ)$, so the antispherical subalgebra
$S^-(\DAHA)$ is isomorphic to the algebra $AW(3,Q_0;qa,qb,c,d)$ by the
map \eqref{58} sending $AW(3,Q_0;qa,qb,c,d)$ to $S^-(\DAHA)$.
\end{Theorem}

The proof is analogous to the proof of Theorem \ref{37}.
Since
$\DAHA$ has a basis \eqref{60},
$S^-(\DAHA)$ is spanned by the elements
\[
(T_1+ab)Z^mY^n(T_1+ab)\qquad(m,n\in\ZZ).
\]
Recall Def\/inition \ref{61}.
Theorem \ref{59} will follow by induction with respect to $|m|+|n|$
from the following analogue of Lemma \ref{62}:
\begin{Lemma}
\label{71}
Let $m,n\in\ZZ$. Then
\[
(T_1+ab)Z^mY^n(T_1+ab)\in
(T_1+ab)\Bigl(\AWQ+o(Z^mY^n)\Bigr)(T_1+ab).
\]
\end{Lemma}
\begin{proof}
We use the same procedure as in the proof of Lemma \ref{62}. I will only
list the main formulas in the three steps. The reader can verify these
formulas in an analogous way as in the proof of Lemma \ref{62}.
\medskip

\noindent
{\bf Step 1}
\begin{gather*}
(T_1+ab)Z^m(T_1+ab) =ab\Bigl(Z^m+Z^{-m}+o(Z^m)\Bigr)(T_1+ab),
\\
(T_1+ab)Z^{-m}(T_1+ab) =-\Bigl(Z^m+Z^{-m}+o(Z^m)\Bigr)(T_1+ab),
\\
(T_1+ab)Y^n(T_1+ab) =-\Bigl(Y^n+(q^{-1}abcd)^nY^{-n}+o(Y^n)\Bigr)(T_1+ab),
\\
(T_1+ab)Y^{-n}(T_1+ab) =ab\Bigl((q^{-1}abcd)^{-n}Y^n+Y^{-n}+o(Y^n)\Bigr)
(T_1+ab),
\\
(T_1+ab)Z^mY^n(T_1+ab) =\Bigl(abZ^mY^n-(q^{-1}abcd)^nZ^{-m}Y^{-n}
+o(Z^mY^n)\Bigr)(T_1+ab),
\\
(T_1+ab)Z^{-m}Y^n(T_1+ab) =\Bigl(-(ab+1)Z^mY^n-(q^{-1}abcd)^nZ^mY^{-n}
-Z^{-m}Y^n
\\
 \phantom{(T_1+ab)Z^{-m}Y^n(T_1+ab) =}{}
+o(Z^mY^n)\Bigr)(T_1+ab),
\\
(T_1+ab)Z^mY^{-n}(T_1+ab) =\Bigl(ab(q^{-1}abcd)^{-n}Z^{-m}Y^n+abZ^mY^{-n}
+(1+ab)Z^{-m}Y^{-n}
\\
 \phantom{(T_1+ab)Z^mY^{-n}(T_1+ab) =}{}
+o(Z^mY^n)\Bigr)(T_1+ab),
\\
(T_1+ab)Z^{-m}Y^{-n}(T_1+ab) =\Bigl(ab(q^{-1}abcd)^{-n}Z^mY^n-Z^{-m}Y^{-n}
+o(Z^mY^n)\Bigr)(T_1+ab).
\end{gather*}
{\bf Step 2}
\begin{gather*}
K_1^m(T_1+ab) =\Bigl(Z^m+Z^{-m}+o(Z^m)\Bigr)(T_1+ab),
\\
K_0^n(T_1+ab) =\Bigl(Y^n+(q^{-1}abcd)^nY^{-n}+o(Y^n)\Bigr)(T_1+ab),
\\
K_1^mK_0^n(T_1+ab) =\Bigl(Z^mY^n+Z^{-m}Y^n
+(q^{-1}abcd)^n\bigl(Z^m Y^{-n}+Z^{-m}Y^{-n}\bigr)
\\
 \phantom{K_1^mK_0^n(T_1+ab) =}{}
+o(Z^mY^n)\Bigr)(T_1+ab),
\\
K_1^{m-1}K_0K_1K_0^{n-1}(T_1+ab) =\Bigl(qZ^mY^n+q^{-1}Z^{-m}Y^n
+q^{-1}(q^{-1}abcd)^nZ^m Y^{-n}
\\
 \qquad{}{}
+(qab)^{-1}(q^{-1}abcd)^n(1+ab-q^2)Z^{-m}Y^{-n}
+o(Z^mY^n)\Bigr)(T_1+ab).
\end{gather*}
{\bf Step 3}
\begin{gather*}
(T_1+ab)Z^mY^n(T_1+ab)=
\frac{ab}{1-q^2}\Bigl(K_1^{m-1}(K_1K_0-qK_0K_1)K_0^{n-1}+o(Z^mY^n)\Bigr)
(T_1+ab),\\
(T_1+ab)Z^{-m}Y^n(T_1+ab)=
\frac 1{1-q^2}\Bigl(K_1^{m-1}(-\bigl(1+ab-q^2)K_1K_0+
qabK_0K_1\bigr)K_0^{n-1}\\
\phantom{(T_1+ab)Z^{-m}Y^n(T_1+ab)=}{}
+o(Z^mY^n)\Bigr)(T_1+ab),\\
(T_1+ab)Z^mY^{-n}(T_1+ab)=
\frac{qab}{(1-q^2)(q^{-1}abcd)^n}
\Bigl(K_1^{m-1}(-qK_1K_0+K_0K_1\bigr)K_0^{n-1}\\
\phantom{(T_1+ab)Z^mY^{-n}(T_1+ab)=}{}
+o(Z^mY^n)\Bigr)(T_1+ab),\\
(T_1+ab)Z^{-m}Y^{-n}(T_1+ab)=
\frac {ab}{(1-q^2)(q^{-1}abcd)^n}
\Bigl(K_1^{m-1}(K_1K_0-qK_0K_1\bigr)K_0^{n-1}\\
\phantom{(T_1+ab)Z^{-m}Y^{-n}(T_1+ab)=}{}
+o(Z^mY^n)\Bigr)(T_1+ab).
\tag*{\qed}
\end{gather*}
\renewcommand{\qed}{}
\end{proof}

\begin{Remark}\rm
The same $q$-shift for the parameters as in Proposition \ref{69} occurs
in \cite[(3.17)]{17} (the eigenfunction of $Y+q^{-1}abcdY^{-1}$ in
$\Asym^-$ expressed in terms of Askey--Wilson polynomials). Of course,
these two results are very much related to each other.

Let $\DAHA(qa,qb,c,d)$ be $\DAHA$ with parameters
$a$, $b$, $c$, $d$ replaced by $qa$, $qb$, $c$, $d$, respectively.
If we compare Theorems \ref{37} and \ref{59} then we can conclude
that the spherical subalgebras $S^-(\DAHA)$ and
$S(\DAHA(qa,qb,c,d))$
are isomorphic. This result is an analogue
of the result in \cite[Proposition 4.11]{19} for Cherednik algebras.
\end{Remark}

\begin{Remark}\rm
\label{74}
Just as we had in Step 1 of the proofs of Lemma \ref{62} and Lemma \ref{71},
we can derive from the f\/irst, fourth and f\/ifth relation in \eqref{39} that
\begin{gather*}
(T_1+1)(Y+q^{-1}a^2b^2cdY^{-1}-(q^{-1}abcd+ab))(T_1+1)=0,\\
(T_1+ab)(Y+q^{-1}cdY^{-1}-(q^{-1}cd+1))(T_1+ab)=0.
\end{gather*}
It follows that, in the basic representation of $\DAHA$, the operator
\[
D^-:=Y+q^{-1}a^2b^2cdY^{-1}-(q^{-1}abcd+ab)
\]
maps $\Asym$ into $\Asym^-$, while
\[
D^+:=Y+q^{-1}cdY^{-1}-(q^{-1}cd+1)
\]
maps $\Asym^-$ into $\Asym$. Since both operators preserve the eigenspace
of $\la_n$, we obtain
\begin{gather*}
D^-\,P_n =ab(q^{-1}cd-q^{-n})(q^n-1)\,Q_n,\\
D^+\,Q_n =(q^{-1}cd-q^{-n}a^{-1}b^{-1})(q^nab-1)\,P_n,
\end{gather*}
where the Askey--Wilson polynomial $P_n$ and
the shifted Askey--Wilson polynomial $Q_n$ are given by \eqref{25} and \eqref{72}, respectively, and the constant factors in the above identities
follow by comparing coef\/f\/icients of $z^n$. Thus the operators $D^-$ and $D^+$
can be considered as {\em shift operators}.
\par
The observations in this Remark were earlier made (in dif\/ferent notation)
in \cite[Lemma 12.2, Proposition 12.3]{11}.
\end{Remark}

\section{Centralizers and centers}
\label{67}
As a corollary of Theorem \ref{37} and Theorem \ref{59} we obtain:
\begin{Theorem}
The centralizer $Z_{\DAHA}(T_1)$ is equal to $\AWQ$.
\end{Theorem}
\begin{proof}
Write $U\in\DAHA$ as
\begin{gather}
U=(1-ab)^{-1}U(T_1+1)+(ab-1)^{-1}U(T_1+ab).
\label{63}
\end{gather}
Suppose that $U\in Z_{\DAHA}(T_1)$. Then
\begin{gather*}
U(T_1+1) =(1-ab)^{-1}(T_1+1)U(T_1+1),\\
U(T_1+ab) =(ab-1)^{-1}(T_1+ab)U(T_1+ab).
\end{gather*}
So $U(T_1+1)\in S(\DAHA)=S(\AWQ){\subset}\AWQ$ and
$U(T_1+ab)\in S^-(\DAHA)=S^-(\AWQ)$%
$\subset\AWQ$.
\end{proof}

The following theorem is interesting in its own right, but it can also
be used, in combination with Theorems \ref{37} and \ref{59}, in order to
show that the center of $\DAHA$ consists of the scalars (see Theorem~\ref{68}).
The proof is in the same spirit as the proof of the faithfulness of
the basic representation of $AW(3,Q_0)$ (see \cite[Theorem 2.2]{17}).
\begin{Theorem}
\label{64}
The center of the algebra $AW(3,Q_0)$ consists of the scalars.
\end{Theorem}
\begin{proof}
Let $U$ be in the center of $AW(3,Q_0)$.
Because of Theorem \ref{12} we may consider $AW(3,Q_0)$ in its
faithful basic representation on $\Asym$. Then $U$ can be uniquely
expanded in terms of the basis of $AW(3,Q_0)$ in this representation:
\begin{gather}
U=\sum_{k,l} a_{k,l} \Dsym^l (Z+Z^{-1})^k+
\sum_{k,l} b_{k,l} \Dsym^{l-1} (Z+Z^{-1}) \Dsym (Z+Z^{-1})^{k-1}.
\label{14}
\end{gather}
Since $U$ is in the center, we have
\begin{gather}
U(Z+Z^{-1})-(Z+Z^{-1})U=0.
\label{15}
\end{gather}
We will f\/irst show that if $U$ given by \eqref{14}
satisf\/ies \eqref{15}, then all coef\/f\/icients $a_{k,l}$ and $b_{k,l}$
in~\eqref{14}
vanish except possibly for coef\/f\/icients $a_{k,l}$ with $l=0$.
Indeed, suppose that this is not the case. Then there is a highest value
$m$ of $k$ for which $a_{k,l}\ne0$ or $b_{k,l}\ne0$ for some $l\ge 1$.
All terms in~\eqref{14} with $k>m$ then will certainly commute with
$Z+Z^{-1}$, so we may assume that the terms in~\eqref{14} with $k>m$ vanish
while \eqref{15} still holds. Let both sides of~\eqref{15} act on the
Askey--Wilson polynomial $P_j[z]$:
\[
(U(Z+Z^{-1})-(Z+Z^{-1})U)P_j[z]=0.
\]
Expand the \LHS\ of the above equation in terms of Askey--Wilson polynomials
$P_i[z]$. Then the highest occurring term will be for $i=j+m+1$, so the
coef\/f\/icient of $P_{j+m+1}[z]$ in this expansion must be zero. This gives
\begin{gather}
\sum_l (a_{m,l}\la_{j+m+1}^l+b_{m,l}\la_{j+m+1}^{l-1}\la_{j+m})
-\sum_l (a_{m,l} \la_{j+m}^l+b_{m,l} \la_{j+m}^{l-1}\la_{j+m-1})=0.
\label{13}
\end{gather}
We have, writing $x:=q^{j+m}$ and $u:=q^{-1}abcd$,
\[
\la_{j+m+1}=q^{-1}x^{-1}+qux,\qquad
\la_{j+m}=x^{-1}+ux,\qquad
\la_{j+m-1}=qx^{-1}+q^{-1}ux.
\]
We can consider the identity \eqref{13} as an
identity for Laurent polynomials in $x$.
Since the left-hand side vanishes
for inf\/initely many values of $x$,
it must be identically zero.
Let $n$ be the maximal $l>0$ for which $a_{m,l}\ne0$ or $b_{m,l}\ne0$.
Then, in particular, the coef\/f\/icients of~$x^{-n}$ and~$x^n$
in the
left-hand side of \eqref{13} must be zero.
This gives explicitly:
\[
a_{m,n}u^n(1-q^n)+q^{-1}b_{m,n}u^n(1-q^n)=0,\qquad
a_{m,n}(1-q^{-n})+q b_{m,n}(1-q^{-n})=0.
\]
Now $n>0$, so $q^{\pm n}\ne1$. Also $u\ne0$. Hence,
\[
a_{m,n}+q^{-1}b_{m,n}=0,\qquad
a_{m,n}+qb_{m,n}=0.
\]
Thus $a_{m,n}=0=b_{m,n}$, which is a contradiction.
\par
So $U$ in the center will have the form
$U=\sum_k a_k (Z+Z^{-1})^k$
in the basic representation of $AW(3,Q_0)$. We have to show that $a_k=0$
for $k>0$. Suppose not. Then there is a highest value $m>0$ of $k$ for
which $a_k\ne0$. Then we have
\[
U\Dsym(1)-\Dsym\,U(1)=0.
\]
Expand the \LHS\ of the above equation in terms of Askey--Wilson polynomials
$P_i[z]$. Then the coef\/f\/icient of $P_m[z]$ will be zero.
So $a_m\la_0-a_m\la_m=0$. Since $\la_m\ne \la_0$ if $m\ne0$, we conclude
that $a_m=0$, a contradiction.
\end{proof}
\begin{Theorem}
\label{68}
The center $Z(\DAHA)$ of $\DAHA$ consists of the scalars.
\end{Theorem}
\begin{proof}
Let $U\in Z(\DAHA)$.
Then $U\in Z_{\DAHA}(T_1)=\AWQ$.
So
$U\in Z(\AWQ)$.
Write $U$ as in~\eqref{63}. We have to show that $U(T_1+1)$ and $U(T_1+ab)$
are scalars. This follows from
\begin{gather*}
 U(T_1+1)=(1-ab)^{-1}S(U)\in Z(S(\DAHA)),\\
 U(T_1+ab)=(ab-1)^{-1}S^-(U)\in Z(S^-(\DAHA)).
\end{gather*}
Now use the algebra isomorphisms from Theorems \ref{37} and \ref{59},
and apply Theorem \ref{64}.
\end{proof}

\subsection*{Acknowledgements}
I thank Jasper Stokman for suggesting me that the
spherical subalgebra of the Askey--Wilson DAHA is related to Zhedanov's algebra.
I thank a referee for suggestions which led to inclusion of Remarks
\ref{75}, \ref{73} and \ref{74}.
Some of the results presented here were obtained during the
workshop {\em Applications of Macdonald Polynomials}, September 9--14, 2007
at the Banf\/f International Research Station (BIRS).
I thank the organizers for
inviting me.

\pdfbookmark[1]{References}{ref}
\LastPageEnding
\end{document}